\documentclass[12pt]{amsart}
\usepackage{amsmath,amsthm,amssymb,mathrsfs}
\usepackage[applemac]{inputenc}

\newcommand{\fr}{\mathcal{F}}

\newcommand{\C}{\mathbb{C}}
\newcommand{\R}{\mathbb{R}}
\newcommand{\N}{\mathbb{N}}

\numberwithin{equation}{section}
\newtheorem{theorem}{Theorem}[section]
\newtheorem{lemma}[theorem]{Lemma}

\theoremstyle{remark}

\newtheorem{example}[theorem]{Example}
\newtheorem{definition}[theorem]{Definition}

\makeatletter
\@namedef{subjclassname@2010}{%
  \textup{2010} Mathematics Subject Classification}
\makeatother

\begin{document}

\title[A  Criterion for quasinormality in $\C^n$ ]{ A  Criterion for quasinormality in $\C^n$}

\thanks{The research work of the first author is supported by research fellowship from UGC( India). The research of the second author is supported by Minor Research Project grant of UGC(India). }

\author[G. Datt]{Gopal Datt}
\address{Department of Mathematics, University of Delhi,
Delhi--110 007, India} \email{ggopal.datt@gmail.com }

\author[S. Kumar]{Sanjay Kumar}

\address{Department of Mathematics, Deen Dayal Upadhyaya College, University of Delhi,
Delhi--110 015, India }
\email{sanjpant@icloud.com; sanjpant@gmail.com}

\begin{abstract}
In this article, we give  a Zalcman type renormalization result  for the quasinormality of a family of holomorphic functions on a domain in  $\C^n$ that takes values in  a complete complex Hermitian manifold.
\end{abstract}
\keywords{analytic set, holomorphic mapping, normal family, quasi-normal family.}
\subjclass[2010]{32A19}
\maketitle
\section{Introduction}
The convergence of a family of functions  always has  far reaching consequences. In his path breaking paper of 1907 ~\cite{M07},  Montel gave a result on the convergence of the family of holomorphic functions which states that a sequence of uniformly bounded holomorphic functions has a subsequence that is locally uniformly convergent. Later in 1912 (see ~\cite{M12}),  he introduced the term {\it normal family} for a family satisfying  this convergence property. In a subsequent paper ~\cite{M22},  he  introduced the notion of quasinormality of a family of functions in one complex variable. All these ideas are well documented in his influential book ~\cite{M27}.  The normality of a family of functions is one of the most fundamental concepts in  function theory of one and several complex variables. It has been extensively used in the study of dynamical properties of functions of one or more complex variables.  In fact, normality plays a vital role in the  Julia-Fatou dichotomy in complex dynamics. In a different direction Beardon and Minda in ~\cite{BM} discuss normal families in terms of maps that satisfy certain types of uniform Lipschitz conditions with respect to various conformal metrics and more background materials can be found in ~\cite{CTC, M27, JLS}. While all theses provide sufficient conditions for normality,  Zalcman in ~\cite{Zalc} proved a striking result that studies consequence of  non-normality. Roughly speaking, it says that in an infinitesimal scaling the family gives a non-constant entire function under the compact-open topology. We  state this renormalization result which has now came to be known as {\it Zalcman's Lemma}: \\

{\bf{Zalcman's Lemma}:} {\it A family $\mathcal F$ of functions meromorphic $($analytic$)$ on the unit disc $\Delta$ is not normal if and only if there exist
\begin{enumerate}
\item[$(a)$]{a number r, $0<r<1$}
\item[$(b)$]{points $z_j, |z_j|<r$}
\item[$(c)$]{functions $\{f_j\}\subseteq \mathcal F$}
\item[$(d)$]{numbers $\rho_j\rightarrow0^+$}
 \end{enumerate}
such that
\begin{equation}\notag
f_j(z_j+\rho_j\zeta)\rightarrow g(\zeta)
\end{equation}
spherically uniformly $($uniformly$)$ on compact subsets of $\C$, where $g$ is a non-constant meromorphic $($entire$)$ function on $\C$.  }\\

This lemma leads to a heuristic principle in function theory. The principle says that any property which forces an entire function to be a constant will also force a family of holomorphic functions to be normal. The source is Marty's inequality which gives a necessary and sufficient condition for the normality of a family of holomorphic or meromorphic functions on a domain $\Omega\subset\C.$\\

 It is very natural to explore the extension of Zalcman's Lemma in several complex variables.   In \cite{AK},  Aladro and  Krantz gave an analogue of Zalcman's Lemma for families of holomorphic mappings  from  a hyperbolic domain of $\C^n$ into complete complex Hermitian manifold $M$ (see also Lemma 5.1~\cite{Iv}). Their analysis was completed by Thai, Trang and Huong in ~\cite{DDT} which addresses the possibility of compact divergence of  the renormalized map $g_j{(\zeta)}=f_j(z_j+\rho_j\zeta)$.  In the same paper ~\cite{DDT} Thai et. al. also defined the concept of Zalcman space. Loosely speaking a complex space $X$ is  Zalcman space if for each non-normal family of holomorphic mappings of unit disc $\{z\in \C: |z|<1\}$ into $X$, we get a non-constant holomorphic mapping  $g:\C\rightarrow X$ under the compact-open topology after an infinitesimal scaling.  This work is further studied in ~\cite{Thai Duc Thu 15,Trao Trang 07}. In this paper, our goal is to prove an analogue of Zalcman's lemma for quasi-normal families in several complex variables. We have illustrated our results with examples. \\

 The theory of quasinormality is well studied in one complex variable.  Chuang, in his text ~\cite{CTC}, introduced the notion of $Q_m-$normality ($m\geq 0$) as an extension of quasinormality in complex plane, $Q_0$ and $Q_1$-normality are usual normality and quasinormality respectively. Loosely speaking a $Q_m-$normal family on a domain $D$ is  normal outside a subset of $D$ whose $m^{\text{th}}$ order derived set is empty.  He introduced the notion of $\mu_m-$point and established some characterizations of $Q_m-$normality. Roughly speaking a point $z_0\in D$ is a $\mu_0-$point of a family $\fr$ if the family  violates the Marty's Criterion on $z_0$ and $\mu_1-$point is the accumulation point of $\mu_0-$points.  Inductively a $\mu_m-$point is an accumulation point of $\mu_{m-1}-$points.  In this paper we extend the notions of $\mu_1$ and $\mu_2-$points in higher dimensions  whereas we  could not generalize the notion of $\mu_m-$points for $m\geq 3$ in several variables due to the nature of zeros of holomorphic mappings in higher dimensions. It seems that the 'order of quasinormality' as given in one variable is not plausible  in higher dimension. It is interesting to note here, using the  notion of $\mu_m-$points,  Nevo proved a Zalcman type renormalization result for $Q_m-$normal families on planar domains  ~\cite{Nevo}. \\

  In several complex variables, the theory of quasinormality has its origin in  the work of  Rutishauser ~\cite{Rut} and  Fujimoto ~\cite{Fuji}.  In ~\cite{Fuji} Fujimoto extending the work of Rutishauser introduced the notion of  meromorphic convergence. In a recent article ~\cite{Iv},  Ivashkovich and  Neji discuss several notions of convergence namely strong convergence, weak convergence and gamma convergence. It can be seen easily from the definitions that weakly-normal implies quasi-normal. In this paper we have also given a renormalization result for weakly-normal family of holomorphic mappings. It is instructive to note here a survey article ~\cite{DD} by  Dujardin where he gives a sufficient condition for quasinormality of a familly of holomorphic functions from a complex manifold to a compact K\"ahler manifold in terms of a suitable sequence of bidegree (1,1) currents. \\

  \section{Preliminary Definitions and Main Results}
Let $\Omega \subset\C^n$ be an open domain and $\Delta$ be the unit disc in $\C$. If $z\in \Omega$ and $\xi\in\C^n$ then -- by the work of Royden \cite{Roy74} -- the infinitesimal form of the Kobayashi pseudo-metric for $\Omega$ at $z$ in the direction $\xi $ is defined  as:

\begin{align*}
F_{K}^{\Omega}(z, \xi)= \inf_f \bigg \{\frac{\|\xi\|}{\|f'(0)\|}:  f: &\  \Delta\rightarrow  \Omega  \text{ is holomorphic, }  f(0)=z,\\
& \text{and} \  f'(0) \text{ is a constant multiple of } \xi \bigg  \},
\end{align*}
where $\|.\|$ represents Euclidean length. And
The Kobayashi pseudo-distance between $z$ and $w$ in $\Omega$ is defined  as:$$K_{\Omega}(z,w)=\inf_{\gamma}\int_{0}^{1}F_{K}^{\Omega}(\gamma(t), \gamma'(t))dt,$$
where the infimum is taken over $C^1$- curves $\gamma:[0, 1]\rightarrow \Omega$ such that $\gamma(0)=z$ and $\gamma(1)=w.$\\

In this work we shall use the following definition of (Kobayashi) hyperbolicity  which -- as shown by Royden \cite{Roy74} -- is equivalent to the original definition.
\begin{definition}\cite{AK, Roy74}
A domain $\Omega \subseteq\C^n$ is called {\it hyperbolic} at a point $z\in\Omega$ if there is a neighborhood $V$ of $z$ in $\Omega$ and a positive constant $c$ such that \begin{center} $F_K^{\Omega}(y, \xi)\geq c\  \|\xi\|$ for all $y\in V$ and all $\xi\in \C^n.$ \end{center}
We say that  $\Omega$ is {\it hyperbolic } if it is hyperbolic at each point.

\end{definition}

Let $M$ be a complete complex Hermitian manifold of dimension $k$ and let $\mathscr{T}_p(M)$ denotes the complexified tangent space to $M$ at $p$. We denote the metric for $M$ at $p$ in the direction of the vector $\xi\in\mathscr{T}_p(M)$ by $E_M(p; \xi).$ Let $\Omega\subseteq\C^n$ be a hyperbolic domain. We denote the set of all holomorphic functions from $\Omega$ into $M$ by Hol$(\Omega, M).$

\begin{definition}
Let $\fr$ be a family of holomorphic mappings of a domain $\Omega$ in $\C^n$ into a complete complex manifold $M$. $\fr$ is said to be a {\it normal family} on $\Omega$ if $\fr$ is relatively compact in Hol$(\Omega, M)$ in the compact open topology.
\end{definition}
\begin{definition}
Let $X, Y$ be complex spaces and $\fr\subset$ {Hol}$(X, Y)$.  A sequence $\{f_j\}\subset \fr$ is {\it  compactly divergent } if for every compact  $K\subset X$ and for every compact  $L\subset Y$ there is a number $J=J (K, L)$ such that $f_j(K)\cap L=\emptyset$ for all $j\geq J.$
If $\fr$ contains no compactly divergent sequences then $\fr$  is called {\it not compactly divergent.}
\end{definition}

Let $\Omega\subseteq\C^n$ be a domain. A subset $S$ of $\Omega$ is called a {\it complex analytic subset} if for any $z\in\Omega$ there exist a neighborhood $U$ of $z$ and holomorphic functions $f_1,\ldots, f_l$ on $U$ such that $S\cap U=\{z\in U: f_1(z)=\ldots=f_l(z)=0\}$. Notice that analytic subsets are closed and nowhere dense  in $\Omega.$

\begin{definition}
A sequence $\{f_j\}$ of holomorphic mappings from a domain $\Omega\subset \C^n$ into  a complete complex Hermitian manifold $M$ is said to be {\it weakly-regular} on $\Omega$ if any $z\in \Omega$ has a connected neighborhood $U$ with the property that $\{f_j(z)\}$ converges uniformly on compact subsets of $U\setminus E$ or compactly diverges on $U\setminus E,$ where $E\subset U$ is  an  analytic subset of codimension at least $2$.
\end{definition}

\begin{definition}
Let $\fr$ be a family of holomorphic mappings from a domain $\Omega$ in $\C^n$ into  a complete complex Hermitian manifold $M$. $\fr$ is said to be a {\it weakly-normal family}  on $\Omega$ if any sequence  in $\fr$ has a weakly-regular subsequence  on $\Omega.$
\end{definition}
\begin{definition}
A sequence $\{f_j\}$ of holomorphic mappings from a domain $\Omega\subset \C^n$ into  a complete complex Hermitian manifold $M$ is said to be {\it quasi-regular} on $\Omega$ if any $z\in \Omega$ has a connected neighborhood $U$ with the property that $\{f_j(z)\}$ converges uniformly on compact subsets of $U\setminus E$ or compactly diverges on $U\setminus E,$ where $E\subset U$ is  a proper complex analytic subset of $U$.
\end{definition}

\begin{definition}
Let $\fr$ be a family of holomorphic mappings from a domain $\Omega$ in $\C^n$ into  a complete complex Hermitian manifold $M$. $\fr$ is said to be a {\it quasi-normal family}  on $\Omega$ if any sequence  in $\fr$ has a quasi-regular subsequence  on $\Omega.$
\end{definition}

\begin{theorem}\label{Aladro}\cite{Alad, AK}
Let $\Omega \subseteq \C^n$ be a hyperbolic domain. Let $M$ be a complete complex Hermitian manifold of dimension $k$ with metric $E_M.$ Let $\mathcal F=\{f_{\alpha}\}_{\alpha\in A}\subseteq \emph{Hol}(\Omega, M).$ If the family $\mathcal F=\{f_{\alpha}\}_{\alpha\in A}$ is a normal family then for each compact set $L\Subset \Omega$ $($i.e $L$ is relatively compact in $\Omega$ $)$, there is a constant $C_L$ such that for all $z\in L$ and all $\xi\in \C^n$, it holds that
\begin{equation}\label{eq}
\sup_{\alpha \in A}|E_M(f_{\alpha}(z); (f_{\alpha})_{*}(z).\xi)|\leq C_LF^{\Omega}_K(z, \xi).
\end{equation}
\end{theorem}

Conversely, if \eqref{eq} holds and if for some $p\in\Omega$,  all $f_{\alpha}(p)$ are in some compact set $Q$ of $M$, then $\fr=\{f_{\alpha}\}_{\alpha\in A}$ is a normal family.\\

Aladro and Krantz  gave an extension of the Zalcman's Lemma to the higher-dimensional setting \cite{AK}. A case missing from the analysis in \cite{AK} was provided by Thai et. al. \cite{DDT}. Their result is as follows:
%\begin{theorem}\label{Kr}\cite{AK}
%Let $\Omega \subseteq \C^n$ be a hyperbolic domain and let $M$ be a complete complex Hermitian manifold of dimension $k.$ Let $\mathcal F=\{f_{\alpha}\}_{\alpha\in A}\subseteq \emph{Hol}(\Omega, M).$ The family $\mathcal F$ is not normal if and only if there exist a compact set $K_0\Subset \Omega$ and  sequences $\{p_j\}\subset K_0, \{f_j\}\subset \mathcal F,  \rho_j>0$ and $\rho_j\rightarrow 0^+$ and $\xi_j\subset \C^n$ Euclidean unit vectors, such that
%\begin{center}
%$g_j(\zeta)=f_j(p_j+\rho_j\xi_j\zeta), $ $\quad \zeta \in \C$
%\end{center}
%converges uniformly on compact subsets of $\C$ to a non-constant entire function $g$.
% \end{theorem}

 \begin{theorem}\label{new}\cite{DDT}
 Let $\Omega$ be a domain in $\C^n.$ Let $M$ be a complete complex Hermitian space. Let $\fr\subset \emph{Hol}(\Omega, M).$ Then the family $\fr$ is not normal if and only if there exist sequences $\{p_j\}\subset{\Omega}$ with $\{p_j\}\rightarrow p_0\in \Omega$,  $\{f_j\}\subset\fr$,  $\{\rho_j\}\subset \R $ with $\rho_j>0$ and $\{\rho_j\}\rightarrow 0$ such that
 \begin{equation}\notag
 g_j(\xi)=f_j(p_j+\rho_j\xi), \qquad \xi\in \C^n
 \end{equation}
 satisfies one of the following two assertions:
 \begin{enumerate}
 \item[(i)] {The sequence $\{g_j\}_{j\geq 1}$ is compactly divergent on $\C^m.$ }
\item[(ii)] {The sequence $ \{g_j\}_{j\geq 1}$ converges uniformly on compact subsets of $\C^n$ to a non-constant holomorphic map $g:\C^n\rightarrow M.$}
 \end{enumerate}
 \end{theorem}

 The main result of this paper provides an analogue of the Zalcman's Lemma for the quasi-normal families. Our main result is as follows:
 \begin{theorem}\label{Main Thm 1}
Let $\Omega\subseteq\C^n$ be a hyperbolic domain. Let $M$ be a complete complex Hermitian manifold of dimension $k$. Let $\fr=\{f_{\alpha}\}_{\alpha\in A}\subseteq \ \text{\emph{Hol}}(\Omega, M)$. The family $\fr$ is not quasi-normal if and only if there exist a subset $E\subset\Omega$ which is either a non-analytic subset or  the closure $\overline E$ has non-empty interior and corresponding to each $p\in E$ there exist
\begin{enumerate}
\item[$(a)$] {a sequence of points $\{w_{j,p}\}_{j=1}^{\infty}\subset \Omega$ such that $w_{j, p}\rightarrow p$.}
\item[$(b)$]{a sequence of functions $\{f_j\}\subset\fr,$}
\item[$(c)$]{a sequence of positive real numbers $\rho_{j, p}\rightarrow 0$, such that}
\end{enumerate}
\begin{center}
$g_j(\zeta)=f_j(w_{j, p}+\rho_{j, p}\xi), \quad \xi\in \C ^n\quad (p\in E)$
\end{center}
satisfies one of the following two assertions.
\begin{enumerate}
\item[(i)]{The sequence $\{g_j\}$ is compactly divergent on $\C^n$.}
\item[(ii)]{The sequence $\{g_j\}$ converges uniformly on compact subsets of $\C^n$ to a non-constant holomorphic map  $g_p:\C^n\rightarrow M.$}
\end{enumerate}
\end{theorem}
The following example will elucidate our result.
 \begin{example}
 Consider the family of holomorphic mappings $\{f_n(z_1, z_2)=z_1^n\}$ from $\C^2$ into $\C$ . Clearly    $f_n$ is not normal in $E=\{(z_1, z_2): |z_1|=1\}.$  Therefore $\{f_n\}$ is not quasi-normal in $\C^2$. To see this  fix $0\leq\theta<2\pi$ and consider the  sequences $z_j=e^{\i\theta/j}$, $\rho_j=1/j$. It can be seen easily that $g_j(\zeta)=f_j(z_j+\rho_j\zeta)$ converges to non-constant holomorphic mapping $ e^{\zeta +\i\theta}.$
\end{example}
%\begin{example}
%Let $D=\{z\in\C: |z|<1 \}$. We consider a family of holomorphic mappings $\{f_n\}$ from $D$ into $\mathbb C$, where $f_n(z)=e^{nz}$. Since  $\{f_n\}$ has no subsequence which is convergent at any point in the set  $E=\{z: \Re{z}=0\}$ so   $\{f_n \}$ is not quasi-normal in $D$.
%\end{example}

 %These definitions and results can be found with complete details in \cite{Alad, AK, Fuji, SK, DDT}. For standard references in several complex variables we have followed \cite{SGK,TN}.
 \section{Proof of Main Result}
 Before giving the proof of our main result (Theorem \ref{Main Thm}), we give some definitions and lemmas, whose one dimensional analogue can be found in \cite{CTC,Nevo}. Throughout section 3, $\Omega\subseteq\C^n$ is a hyperbolic domain and $M$ denotes a complete complex Hermitian manifold.
 Here we extend the notions of $\mu_1-$point and $\mu_2-$point of a sequence $\{f_j\}\subset$ {Hol}$(\Omega, M)$.

 \begin{definition}
Let $\Omega \subseteq \C^n$ be a hyperbolic domain. Let $M$ be a complete complex Hermitian manifold of dimension $k$. Consider a sequence  $\{f_j\}\subset \text{Hol}(\Omega, M)$. A point $p_0\in \Omega$ is said to be a {\it  $\mu_1-$point } of $\{f_j\}$, if for each subset $K\Subset \Omega$ containing $p_0,$
$$
\lim_{j\rightarrow\infty}\sup_{p\in K, \|\xi\|=1}|E_M(f_j(p), (f_j)_*(p).\xi)|=\infty.
$$
\end{definition}

\begin{enumerate}
\item{A point  $p_0$ is called a {\it  $\mu_2-$point} of $\{f_j\}$ if there exists an analytic set   $K\subset\Omega$ of codimension at most $1$, containing $p_0$, such that each point of $K$ is a $\mu_1-$point of $\{f_j\}$.}
  \item{We say $p_0$ is a {\it  $q-$point} of $\{f_j\}$ if there exists a subset   $K\subset\Omega$, containing $p_0$, such that closure  $\overline K$ has non-empty interior and each point of $K$ is a $\mu_1-$point of $\{f_j\}$.}
\item{ We say $p_0$ is an {\it $\lambda-$point} of $\{f_j\}$ if there exists a non-analytic subset $K\subset\Omega$ containing $p_0$ such that each point of $K$ is a $\mu_1-$point of $\{f_j\}$.}
\end{enumerate}

\begin{example}
Let $\{f_n\}$ be a family of holomorphic mappings from $\C^2$ on to itself  such that $f_n(z)=n z$, where $z=(z_1, z_2).$  Then $z =(0, 0)$ is a $\mu_1-$point of $\{f_n\}.$
\end{example}
\begin{example}
 Let $\{f_n\}$ be a family of holomorphic mappings defined on the polydisc $D=\{(z_1, z_2) : |z_1|<1 \ { \text{and}\ } |z_2|<1 \}$ such that $f_n(z_1, z_2)= nz_1z_2.$ Then each point of $E=\{(z_1, z_2) : z_1z_2=0\}$ is a $\mu_2-$point of $\{f_n\}.$
\end{example}
\begin{example}
 Let $\{f_n\}$ be a family of holomorphic mappings defined on $\C^2$ such that $f_n(z_1, z_2)=e^{nz_1}.$ Then each point of $E=\{(z_1, z_2) : \Re{z_1}=0\}$ is an $\lambda-$point of $\{f_n\}.$
\end{example}
 \begin{lemma}\label{lem 1}
 Let $\Omega \subseteq \C^n$ be a hyperbolic domain. Let $M$ be a complete complex Hermitian manifold of dimension $k$. A family $\fr\subset \emph{Hol}(\Omega, M)$ is normal in $\Omega$ if and only if each sequence $\{f_j\}$ of $\fr$ has no $\mu_1-$point in $\Omega.$
 \end{lemma}
 \begin{proof}
Suppose $\fr$ is  normal then by Theorem \ref{Aladro} there is no $\mu_1-$point for any sequence $\{f_j\}$ of $\fr$.\\

Conversely, suppose no sequence has a $\mu_1-$point in $\Omega$. Assume, on  the  contrary, that $\fr$ is not normal in $\Omega$. Consider a sequence $\{f_j\}$ of $\fr$, then there is a point $p_0\in \Omega$ such that we can not find a ball $\Gamma =\{p: \|p-p_0\|<r\}$, $\overline{\Gamma}\Subset\Omega$ and a number $N>0$ such that for $j\geq1$ we have
\begin{equation}\notag
|E_M(f_j(p); (f_j)_*(p).\xi)|\leq N \ \text{ in \ } \overline{\Gamma}.
\end{equation}
Take two sequences of positive real numbers $\{r_k\}\rightarrow 0$ and $\{N_k\}\rightarrow \infty$ such that the ball $\overline{\Gamma}_k= \{p: \|p-p_0\|\leq r_k\}$ is contained in $\Omega.$ Then there is an integer $j_1\geq1$ such that
$$\sup_{p\in \overline{\Gamma}_1, \|\xi\|=1}|E_M(f_{j_1}(p); (f_{j_1})_*(p).\xi)|>N_1.$$

Next there is an integer $j_2>j_1$ such that $$\sup_{p\in \overline{\Gamma}_2, \|\xi\|=1}|E_M(f_{j_2}(p); (f_{j_2})_*(p).\xi)|>N_2.$$
Continuing in this manner, we get a sequence of integers $\{j_k\}, (k=1,2,\ldots)$ such that for $k\geq1,$ we have $$\sup_{p\in \overline{\Gamma}_k, \|\xi\|=1}|E_M(f_{j_k}(p); (f_{j_k})_*(p).\xi)|>N_k.$$\\

Now consider a ball $\displaystyle{\Gamma: \|p-p_0\|<r}\ \text{ such that \ }  \overline{\Gamma}\Subset\Omega.$ Let $k_0\geq 1$ be an integer such that $r_k<r$ for $k\geq k_0,$ then for $k\geq k_0$ we have $$N_k < \sup_{p\in \overline{\Gamma}_k, \|\xi\|=1}|E_M(f_{j_k}(p); (f_{j_k})_*(p).\xi)|\leq\sup_{p\in \overline{\Gamma}, \|\xi\|=1}|E_M(f_{j_k}(p); (f_{j_k})_*(p).\xi)|.$$ Hence $$\lim_{k\rightarrow \infty}\sup_{p\in \overline{\Gamma}, \|\xi\|=1}|E_M(f_{j_k}(p); (f_{j_k})_*(p).\xi)|=\infty.$$

This implies $p_0$ is a $\mu_1-$point of the sequence $\{f_j\}$ which is a contradiction.
\end{proof}
\begin{lemma}\label{lem 2}
Let $\Omega \subseteq \C^n$ be a hyperbolic domain. Let $M$ be a complete complex Hermitian manifold of dimension $k$. A family $\fr\subset \emph{Hol}(\Omega, M)$ is weakly-normal in $\Omega$ if and only if each sequence $\{f_j\}$ of $\fr$ has neither a $\mu_2-$point nor a $\lambda-$point in $\Omega.$
\end{lemma}
\begin{proof}
 Suppose that $\fr$ is weakly-normal in $\Omega$. Let $\{f_j\}$ be a sequence of functions of $\fr$.  Then  we can extract a weakly-regular subsequence $\{f_{j_k}\}$ from $\{f_j\}$.  On the contrary we assume that $\{f_j\}$ has a $\mu_2-$point $p_0$ in $\Omega$ . Since $\fr$ is weakly-normal  therefore we can find  a neighborhood $U_0$ of $p_0$ in $\Omega$ such that $\{f_{j_k}\}$ converges uniformly  on compact subsets of $U_0\setminus E,$ or diverges compactly on $U_0\setminus E$, where $E$ is an analytic subset of $U_0$ of codimension at least 2. For each $p'\in U_0\setminus E$, $\{f_{j_k}\}$ converges or diverges compactly hence $\{f_{j_k}\}$ is normal in $U_0\setminus E$ so by Lemma \ref{lem 1}, $\{f_{j_k}\}$ has no $\mu_1-$point in $U_0\setminus E.$ But by definition of $\mu_2-$point there exist an analytic set $K\subset\Omega$ of codimension at most 1, such that each point of $K$ is a $\mu_1-$point of $S$ and hence of  $S'$ which is a contradiction. Similar argument can be given if $p_0$ is $\lambda-$point.\\

 Conversely, suppose that $\fr$ has neither a $\mu_2-$point nor a $\lambda-$point and  $\fr$  is not weakly-normal on $\Omega.$ Consider a set $K\Subset\Omega$.  Let $\{f_j\}$ be a sequence of functions of  $\fr$ then we can not extract a subsequence which is weakly-regular in $K$. Then $\{f_j\}$ must have $\mu_1-$points in $K$, also set $V$ of all $\mu_1-$points contains either a non-empty analytic subset $V_1\subset\Omega$ of codimension at most 1 or a non-analytic set $V_2\subset\Omega$, otherwise $\{f_j\}$ constitutes a weakl-normal family. Since $V_1$ is a set of codimension at most 1 then for $p\in V_1$ there exists a neighborhood $N_1$ of $p$  and each point of $N_1\cap V_1$ is a $\mu_1-$point of $S$, thus $p$ is a $\mu_2-$point of $S$. Also $V_2$ is a non-analytic set then for $p\in V_2$ there exists a neighborhood $N_2$ of $p$  and each point of $N_2\cap V_2$ is a $\mu_1-$ point of $S$, thus $p$ is a $\lambda-$point of $S$. In either case we get a contradiction.
\end{proof}
In the same lines we can prove the following result:
\begin{lemma}\label{lem 3}
Let $\Omega \subseteq \C^n$ be a hyperbolic domain. Let $M$ be a complete complex Hermitian manifold of dimension $k$. A family $\fr\subset \emph{Hol}(\Omega, M)$ is quasi-normal in $\Omega$ if and only if each sequence $\{f_j\}$ of $\fr$ has neither a $q-$point nor a $\lambda-$point in $\Omega.$
\end{lemma}

We now give the  Local version of Zalcman's Lemma for Normal families.
\begin{lemma}\label{local}
Let $\Omega \subseteq \C^n$ be a hyperbolic domain. Let $M$ be a complete complex Hermitian manifold of dimension $k.$ Let $\mathcal F=\{f_{\alpha}\}_{\alpha\in A}\subseteq \emph{Hol}(\Omega, M).$ The family $\mathcal F$ is not normal at $p_0\in \Omega$ if and only if there exist
\begin{enumerate}
\item[$(a)$]{a sequence $\{p_j\}\subset\Omega$ such that $p_j\rightarrow p_0.$}
\item[$(b)$]{a sequence of functions $\{f_j\}\subset\fr,$}
\item[$(c)$]{a sequence of positive real numbers $\rho_j\rightarrow 0$ such that}
\end{enumerate}
\begin{center}
$g_j(\xi)=f_j(p_j+\rho_j\xi), \quad \xi\in \C^n$
\end{center}
satisfies one of the following two assertions.
\begin{enumerate}
\item[(i)]{The sequence $\{g_j\}$ is compactly divergent on $\C^n.$}
\item[(ii)]{The sequence $\{g_j\}$ converges uniformly on compact subsets of $\C^n$ to a non-constant holomorphic map  $g:\C^n\rightarrow M.$}
\end{enumerate}
\end{lemma}
\begin{proof}
Assume  that $\mathcal F$ is not  normal at $p_0$,  then by Theorem ~\ref{Aladro} there exists a compact set $K_0\subset\{p:\|p-p_0\|\leq \rho\}=K_1$ for some $\rho>0$ and a sequence $f_j\subset \fr, \{q_j\}\subset K_0$ and $\{\xi_j\}\subset\C^n,$  such that
\begin{equation}\label{l3}
|E_M(f_j(q_j); (f_j)_*(q_j).\xi_j)|\geq jF_K^{\Omega}(q_j, \xi_j).
\end{equation}
Let $k_0\in \N$ be such that $\displaystyle{\frac{1}{\sqrt {k_0}}<\rho}$,  then for $k\geq k_0$ there is $f_k\in\fr$ with
\begin{equation}\label{l4}
|E_M(f_k(q_k); (f_k)_*(q_k).\xi_k)|\geq kF_K^{\Omega}(q_k, \xi_k),\ \text{ for all\ } k\geq k_0\ \text{and\ } q_i\in\left\{p: \|p-p_0\|\leq \frac{1}{2\sqrt k}\right\}.
\end{equation}

Now define  $$g_k(p)=f_k\left(p_0+\frac{p}{\sqrt k}\right).$$
 Each $g_k$ is defined on $\Delta=\{p: \|p\|< 1\}$ and satisfies
 \begin{align}
|E_M(g_k(q_k); (f_k)_*(q_k).\xi_k|&=\left|E_M\left(f_k\left(p_0+\frac{q_k}{\sqrt k}\right); \frac{1}{\sqrt k}(f_k)_*\left(p_0+\frac{q_k}{\sqrt k}\right).\xi_k\right)\right|\notag\\
&\geq \sqrt k F_K^{\Omega}(q_k, \xi),
 \end{align}
therefore $\{g_k\}$ is not normal in $\Delta$. Now by Theorem \ref{new} there exist
\begin{enumerate}
\item[$(1)$] {a compact set $K\Subset \Delta$,}
\item[$(2)$] {a sequence $\{p_j^*\}\subset K$,}
\item[$(3)$] {a sequence $\{g_{k_j}\}\subset \{g_k\}$,}
\item[$(4)$] {a sequence of positive real numbers $\rho_j^* \rightarrow 0$}
\end{enumerate}
such that $\displaystyle{h_{k_j}(\xi)=g_{k_j}\left(p_j^*+\rho_j^*\xi\right),\  \xi\in \C^n}$ either compactly divergent on $\C^n$ or  converges uniformly on compact subsets of $\C^n$ to a non-constant holomorphic map $g:\C^n\rightarrow M$. This is same as $\displaystyle{f_{k_j}\left(\frac{p_j^*}{\sqrt k_j}+\frac{\rho_j^*}{\sqrt{k_j}}\xi\right), \ \xi \in \C^n}$ either compactly divergent on $\C^n$ or  converges uniformly on compact subsets of $\C^n$ to a non-constant holomorphic map $g:\C^n\rightarrow M$. Now set $$\displaystyle{p_j=\frac{p_j^*}{\sqrt{k_j}}+p_0 \ \text{and\ } \rho_j=\frac{\rho_j^*}{\sqrt{k_j}}}.$$
This proves the necessity part of the Lemma.

Conversely, assume that  the conditions of the lemma are satisfied and suppose, on the contrary, that $\mathcal F$ is normal at $p_0$. Then by analogue of Marty's Theorem in $\C^n$, for compact subsets $K_0$ and $K_1$ with $p_0\in K_0\Subset K_1\Subset\Omega$, there exists a number $N>0$ such that
\begin{equation}\label{l1}
\sup_{p\in K_1,  \|\zeta\|=1}|E_M(f(p); (f)_{*}(p).\zeta)|\leq N, \quad  \text{for each\ } f\in \mathcal F.
\end{equation}

Now, suppose $g_j(\xi)=f_{j}(p_j+\rho_j\xi)$ converges uniformly on compact subsets of $\C^n$ to a non-constant holomorphic map $g:\C^n\rightarrow M$ we have
\begin{align}
|E_M(g_j(\xi); g_j'(\xi).\zeta)|&=|E_M(f_j(p_j+\rho_j\xi); \rho_j(f_j)_{*}(p_j+\rho_j\xi).\zeta)|\notag\\
&\leq \rho_jN. \label{l2}
\end{align}
Taking the limit, we get $$\lim_{j\rightarrow\infty}|E_M(g_j(\xi); g_j'(\xi).\zeta)|=|E_M(g(\xi); g'(\xi).\zeta)|=0.$$ Then $g'(\xi)=0$ for any $\xi\in \C^n,$ therefore $g$ is a constant function which is a contradiction.\\

Next, suppose that $g_j(\xi)=f_{j}(p_j+\rho_j\xi)$ is compactly divergent. Since the family is normal, without any loss of generality, we may assume that the sequence $\{f_j\}\rightarrow f.$ And we get $g_j(\xi)\rightarrow f(p_0)$, which is not possible as $\{g_j\}$ is compactly divergent.
This completes the proof.
\end{proof}
\begin{example}
Let $D=\{(z_1, z_2) : |z_1|<1 \ { \text{and}\ } |z_2|<1 \}$ be the polydisc in $\C^2$. We consider a family of holomorphic mappings $\{f_n\}$, from $D$ into $\mathbb{C}$ where $f_n(z_1, z_2)=e^{nz_1z_2}$, for all $n\in\N$. Since  $\{f_n\}$ has no subsequence which is convergent at any point in the set  $E=\{\{(\Re{(z_1)}, 0)\times (0,\Im{(z_2)})\}\cup \{(0, \Im{(z_1)})\times (\Re{(z_2)}, 0)\}\}\cap D$ so   $\{f_n \}$ is not normal in $D$. As $\{f_n\}$ is not normal at $(0, 0)$,  we get a sequence $\{p_n\}$ in $D$ such that $p_n=\left(\frac{z_1^0}{\sqrt{n}}, \frac{z_2^0}{\sqrt{n}}\right)$, where $(z_1^0, z_2^0)$ is a fixed point in $D$. Notice that $\{p_n\}\rightarrow (0, 0).$ Also we have a sequence of positive real numbers $\{\rho_n\}\rightarrow 0,$ where $\rho_n=\frac{1}{\sqrt n}$  such that for all $\xi=(z_1, z_2)\in \C^2$ we have
\begin{equation}\notag
 g_n(p_n+\rho_n\xi)=f_n(p_n+\rho_n\xi)\rightarrow e^{(z_1^0+z_1)(z_2^0+z_2)}.\end{equation}
 \end{example}
Now we are ready to prove Theorem \ref{Main Thm 1}.
\begin{proof}[The proof of Theorem \ref{Main Thm 1}]
Suppose all the conditions of the theorem  are satisfied. Since $E$ is  is either a non-analytic subset or  the closure $\overline E$ has non-empty interior then for each $p_0\in E$, we can get a sequence $p_j$ in $E$ such that $p_j\rightarrow p_0.$ By Lemma $\ref{local}$, $\fr$ is not normal at $p_0$. Since $p_0$ is an arbitrary point of $E$ and $E$ is either a dense subset  or a non-analytic subset of $\Omega$, $\fr$ is not quasi-normal in $\Omega.$\\

Conversely, suppose $\fr$ is not quasi-normal family in $\Omega$. Then by Lemma \ref{lem 3}, there exists a sequence $S'=\{h_j\}$ of $\fr$ which has either a $q-$point  or a $\lambda-$point $p_0\in \Omega$. This implies that there exists a subset $V\Subset\Omega$ which is either  dense   or a  non-analytic subset containing $p_0$ so that each point of $V$ is a $\mu_1-$point  of $S'$.   Since $V$ is either dense or  non-analytic,  we can choose a sequence of positive real numbers $\{r_i\}$ such that $\{r_i\}\rightarrow 0$ and for each open ball $B(p_0, r_i)=\{p\in \Omega : \|p-p_0\|<r_i\}$, the set $V\cap B(p_0, r_i)$ has at least one $\mu_1-$point.  Now we proceed inductively to get conditions of the theorem.\\
{\it\underline{Step 1.}}  There exists
\begin{enumerate}
\item[$(A_1)$] {a $\mu_1-$point $p_1\in \Omega$ such that $p_1\in V\cap B(p_0, r_1).$ So $S'$ is not normal at $p_1.$ Therefore, by Lemma \ref{local} there exist}
    \item[$(B_1)$]{ a sequence $\{w_{j,1}\}\subset\Omega$ such that $\{w_{j,1}\}\rightarrow p_1,$}
    \item[$(C_1)$]{a subsequence $S_1=\{h_{j,1}\}$ of $S_0$,}
    \item[$(D_1)$]{a sequence of positive real numbers $\{\rho_{j,1}\}\rightarrow 0,$ such that}
   \end{enumerate}
$\displaystyle{h_{j,1}(w_{j,1}+\rho_{j,1}\xi), \ \xi \in \C^n,}$ either compactly divergent on $\C^n$ or  converges uniformly on compact subsets of $\C^n$ to a non-constant holomorphic map $g_1:\C^n\rightarrow M$.\\
{\it\underline{Step 2.}} Since $p_0$ is also a $q-$point or a $\lambda-$point of $S_1$, there exists
\begin{enumerate}
\item[$(A_1)$] {a $\mu_1-$point $p_2\in \Omega$, $p_2\neq p_1,$  such that $p_2\in V\cap B(p_0, r_2), 0<r_2<r_1.$ So $S_1$ is not normal at $p_2.$ Therefore, by Lemma \ref{local} there exist}
    \item[$(B_1)$]{ a sequence $\{w_{j,2}\}\subset\Omega$ such that $\{w_{j,2}\}\rightarrow p_2,$}
    \item[$(C_1)$]{a subsequence $S_2=\{h_{j,2}\}$ of $S_1$,}
    \item[$(D_1)$]{a sequence of positive real numbers $\{\rho_{j,2}\}\rightarrow 0,$}, such that
    \end{enumerate}
$\displaystyle{h_{j,2}(w_{j,2}+\rho_{j,2}\xi), \ \xi \in \C^n,}$ either compactly divergent on $\C^n$ or  converges uniformly on compact subsets of $\C^n$ to a non-constant holomorphic map $g_2:\C^n\rightarrow M$.\\

 Continuing in this manner we get sequences $\{p_j\}\rightarrow p_0$, $\{w_{i,j}\}, \{\rho_{i,j}\}, \{g_j\} $ and $\{h_{i,j}\}$. Now we use the Cantor's diagonal method and choose $E = V$;  $f_i=h_{i,i};   w_{i,p_o}=w_{i,i};   \rho_{i,p_0}=\rho_{i,i}.$ Then for each $j\geq 1$, $\{f_i\}_{i=j}^{\infty}$ is a subsequence of $S_j$ and $f_i(w_{i,p_0}+\rho_{i,p_0}\xi), \xi\in \C^n,$ either compactly divergent on $\C^n$ or  converges uniformly on compact subsets of $\C^n$ to a non-constant holomorphic map $g_{p_0}:\C^n\rightarrow M$. This completes the proof of theorem.

\end{proof}
For the {\it weakly-normal} family we propose the following theorem.
\begin{theorem}\label{Main Thm}
Let $\Omega\subseteq\C^n$ be a hyperbolic domain. Let $M$ be a complete complex Hermitian manifold of dimension $k$. Let $\fr=\{f_{\alpha}\}_{\alpha\in A}\subseteq \ \text{\emph{Hol}}(\Omega, M)$. The family $\fr$ is not weakly-normal if and only if there exist a subset $E\subset\Omega$ which is either an analytic subset of codimension at most 1 or a non-analytic subset and corresponding to each $p\in E$ there exist
\begin{enumerate}
\item[$(a)$] {a sequence of points $\{w_{j,p}\}_{j=1}^{\infty}\subset \Omega$ such that $w_{j, p}\rightarrow p$.}
\item[$(b)$]{a sequence of functions $\{f_j\}\subset\fr,$}
\item[$(c)$]{a sequence of positive real numbers $\rho_{j, p}\rightarrow 0$, such that}
\end{enumerate}
\begin{center}
$g_j(\zeta)=f_j(w_{j, p}+\rho_{j, p}\xi), \quad \xi\in \C ^n\quad (p\in E)$
\end{center}
satisfies one of the following two assertions.
\begin{enumerate}
\item[(i)]{The sequence $\{g_j\}$ is compactly divergent on $\C^n$.}
\item[(ii)]{The sequence $\{g_j\}$ converges uniformly on compact subsets of $\C^n$ to a non-constant holomorphic map  $g_p:\C^n\rightarrow M.$}
\end{enumerate}
\end{theorem}

The proof of   Theorem \ref{Main Thm} is merely a formality. It can be proven on the similar lines as of the proof of Theorem \ref{Main Thm 1} using Lemma \ref{lem 2} instead of Lemma \ref{lem 3}.\\

The following examples elucidate Theorem \ref{Main Thm}.

\begin{example}
 Let $\{f_n\}$ be a family of holomorphic mappings defined on the polydisc $D=\{(z_1, z_2) : |z_1|<1 \ { \text{and}\ } |z_2|<1 \}$ such that $f_n(z_1, z_2)= nz_1z_2.$ Then $\{f_n\}$ is not weakly-normal on $D$, as $\{f_n\}$ converges compactly in $D\setminus E$, where $E=\{(z_1, z_2) : z_1z_2=0\}$ is an analytic subset of codimension 1 of $D$. Let $(z_1^0, z^0_2)\in E$ be  any arbitrary point, without loss of generality we take $z_1^0=0.$ Then we get a sequence $\{p_n\}\rightarrow (0, z_2^0)$ of points in $E$, where $p_n=\left(0, z_2^0+\frac{1}{\sqrt n}\right)$. Also we have a sequence of positive real numbers $\{\rho_n\}\rightarrow 0,$ where $\rho_n=\frac{1}{\sqrt n}$  such that for all $\xi= (z_1, z_2)\in \C^2$ and we get $$g_n(p_n+\rho_n\xi)=f_n(p_n+\rho_n\xi_n)\rightarrow z_2^0 z_1z_2.$$
\end{example}
\begin{example}
Let $\{f_n\}$ be a family of holomorphic mappings defined on the polydisc $D=\{(z_1, z_2) : |z_1|<1 \ { \text{and}\ } |z_2|<1 \}$ such that $f_n(z_1, z_2)= \cos(nz_1z_2).$ Then $\{f_n\}$ is not  weakly-normal on $D$ as $\{f_n\}$ is not compactly convergent in any open subset of $D$ containing $E=\{(z_1, z_2) : z_1z_2=0\},$  which is of codimension 1. Let $(z_1^0, z^0_2)$ be  any arbitrary point of $E$, without loss of generality we take $z_2^0=0.$ Then we get a sequence $\{p_n\}\rightarrow (z_1^0, 0)$ of points in $E$, where $p_n=\left( z_1^0+\frac{1}{\sqrt n}, 0\right)$. Also we have a sequence of positive real numbers $\{\rho_n\}\rightarrow 0,$ where $\rho_n=\frac{1}{\sqrt n}$ such that for all $\xi= (z_1, z_2)\in \C^2$ and we obtain $$g_n(p_n+\rho_n\xi)=f_n(p_n+\rho_n\xi)\rightarrow \cos(z_1^0z_1z_2).$$\\
\end{example}
Before ending the paper we give one more example of a family of holomorphic mappings which is not quasi-normal.
\begin{example}
Let $D=\{z\in\C: |z|<1 \}$. We consider a family of holomorphic mappings $\{f_n\}$ from $D$ into $\mathbb C$, where $f_n(z)=e^{nz}$. Since  $\{f_n\}$ has no subsequence which is convergent at any point in the set  $E=\{z: \Re{z}=0\}$ so   $\{f_n \}$ is not quasi-normal in $D$.
\end{example}
{\bf{Acknowledgement:}} We  would like to thank  Kaushal Verma and  Gautam Bharali, IISc Bangalore for stimulating discussions  about this work as well as for  critical comments.

\end{document}